\documentclass[11pt]{article}
\usepackage{amssymb,amsfonts,amsmath,latexsym,epsf,tikz,url}
\usepackage[usenames,dvipsnames]{pstricks}
\usepackage{pstricks-add}
\usepackage{epsfig}
\usepackage{pst-grad} 
\usepackage{pst-plot} 
\usepackage[space]{grffile} 
\usepackage{etoolbox} 
\makeatletter 
\patchcmd\Gread@eps{\@inputcheck#1 }{\@inputcheck"#1"\relax}{}{}
\makeatother

\newtheorem{theorem}{Theorem}[section]

\newtheorem{corollary}[theorem]{Corollary}
\newtheorem{lemma}[theorem]{Lemma}

\newtheorem{example}[theorem]{Example}
\newtheorem{definition}[theorem]{Definition}

\newcommand{\qed}{\hfill $\square$\medskip}

\begin{document}

\title{Golden ratio in graph theory: A survey }

\author{ 
Saeid Alikhani$^{}$\footnote{Corresponding author} \and Nima Ghanbari 
}

\date{\today}

\maketitle

\begin{center}

Department of Mathematical Sciences, Yazd University, 89195-741, Yazd, Iran\\

\medskip
\medskip
{\tt  ~~
alikhani@yazd.ac.ir, n.ghanbari.math@gmail.com
 }

\end{center}

\begin{abstract}
Much has been written about the golden ratio $\phi=\frac{1+\sqrt{5}}{2}$ and this strange number appears mysteriously in many mathematical calculations. In this article, we review the appearance of this number in the graph theory. More precisely, we review the relevance of this number in topics such as  the number of spanning trees, topological indices, energy, chromatic roots,  domination roots and  the number of domatic partitions of graphs. 
 
\end{abstract}

  
\medskip
\noindent{\bf AMS Subj.\ Class.}:  05C69.

\section{Introduction}

The graph $G$ is the ordered pair $\big (V(G) ,E(G)\big )$ where $ V(G) $ is  the set of elements called vertices and $E(G)$ is a finite set of pairs of distinct elements $V(G)$ called set of edges. Two vertices such as $v_1$ and $v_2$ are called adjacent, whenever $\{v_1,v_2\}\in E(G) $. The vertex $x$ is called a common neighbor of $v_1$ and $v_2$, whenever $x$ with both vertices
$v_1$ and $v_2$ are adjacent. The degree of the vertex $v$ is the number of edges connected to the vertex $v$ and it is denoted by the symbol $deg(v)$ or $d(v)$. The vertex $v$ is called isolate,  if
$d(v)=0$. The minimum and maximum degrees of the vertices of the graph $G$ are denoted by   $\delta(G)$
and $\Delta(G)$, respectively.  The graph $H$ is called a subgraph of the graph $G$, whenever $V(H)\subseteq V(G)$ and $E(H)\subseteq E(G)$.
If we have $V(H)=V(G)$ for the subgraph $H$, then we call $H$ a spanning subgraph of $G$. A spanning tree of a graph $G$ is a spanning subgraph $G$ which is also a tree. The number of spanning trees of a graph $G$, usually denoted by the symbol $\tau(G)$, is an important parameter in graph theory that has many applications, including in chemistry and nanotechnology.
The sequence $F_n$ of natural numbers defined by the equations $F_0 = 0$, $F_1 = 1$ and $F_n = F_{n-1}+F_{n-2}$ $(n\geq 2)$ is called the Fibonacci sequence. The $n$-th term in the sequence is called the $n$-th Fibonacci number. The Lucas sequence is denoted by $L_n$ is $L_{n}=L_{n-1}+L_{n-2}$ for $n\geq 3$ with initial conditions $L_1=1$ and $L_2  = 3$.
In mathematics, two quantities have the golden ratio (often denoted by the symbol $\phi$) if their ratio is equal to the ratio of their sum to the larger quantity. This property can be expressed algebraically as follows when $a>b>0$:
\[
\frac{a+b}{a}=\frac{a}{b}=\phi.
\]
This ratio is an irrational number, which is also a positive answer for the equation $x^2-x-1=0$. One of the most beautiful ways to reach this golden ratio is using Fibonacci numbers in the form of $
{\lim}_{n\rightarrow \infty}\frac{F_{n+1}}{F_n}$. Much has been said and written about this interesting and surprising number \cite{Koshy}.

\medskip 
In this paper, we review the relevance of the golden ratio  in topics such as the number of spanning trees, topological indices, chromatic roots,  domination roots and  the number of domatic partitions of graphs.


\section{The number of spanning trees and $\phi$}
Spanning graph trees and their number have been of interest to mathematicians for a long time, and interestingly, the mentioned parameter has many applications in chemical and nano-mathematical sciences. In this section, we examine the relationship between the number of spanning trees of some graphs and the golden ratio.
The number of spanning trees of a graph can be a large number, for example the famous Petersen graph has $2000$ labeled spanning trees \cite{Koshy}.

Let $\tau(G)$ be the number of spanning trees of graph $G$. Therefore $\tau(G) \geq 1$ if and only if $G$ is connected and $\tau(G) = 1$ if and only
If $G$ is a tree.
The edge $e$ of the graph $G$ is said to be contracted, whenever the edge $e$ of the graph $G$ is removed and the terminals of this edge coincide. The new graph created under this operation is the contraction of the edge $e$ in $G$ which is denoted by $G\circ e$. A recursive relation to find the number of spanning trees of an arbitrary graph $G$ is as follows.

\begin{theorem}{\rm\cite{West}}\label{span}
	If $e$ is an edge of the graph $G$, then $\tau(G) = \tau(G - e) + \tau(G\circ e)$.
\end{theorem}
The recurrence relation $\tau(G)$ is computationally long and somewhat tedious despite its interestingness. The tree-matrix theorem provides another way to calculate $\tau(G)$ using linear algebra.
\vspace{-.2cm}
\begin{theorem} {\rm \cite{West}} { (tree-matrix) }
	Suppose $G$ is a graph with order $n\geq 2$ and set of vertices $ \{ v_1 , v_2 , \ldots , v_n \}$. In this case, $\tau(G)$ is equal to the cofactor  of each entry  of the matrix $D(G) - A(G)$, where $D(G)$ is a diagonal matrix.
	\[
	{(d_{ij})}_{n \times n}=\left\{\begin{array}{cc} {\deg (v_i)} & \quad{ i= j}\\[15pt]
	{0} & \quad{ i \not= j}
	\end{array} \right.
	\]
	and $A(G)$ is the adjacency matrix of $G$.
\end{theorem}

Klitman and Golden in 1975 obtained the number of spanning trees of the second power of cycle $C_n$ as follows \cite{Kle}. We remind that the quadratic of a graph $G$ is a graph obtained by adding edges between vertices whose distance is equal to two. For example, the square of the graph $C_4$ is equal to the graph $K_4$.

\begin{theorem}{\rm\cite{Kle}}
	The number of spanning trees of the graph $C_n^2$ for $n\geq 5$ is equal to
	$\tau(C_n^2)=nF_n^2$.
\end{theorem}

\begin{figure}[h]
	\begin{center} 
	\includegraphics[width=0.5\textwidth]{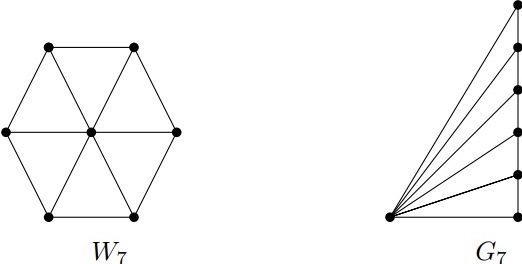}
	\caption{wheel graph $W_7$ and fan graph $G_7$}\label{fan}
	\end{center} 
\end{figure}

Joining two graphs $G_1=(V_1,E_1)$ and $G_2=(V_2,E_2)$
is denoted by $G_1 \vee G_2$ and  is the graph with the set of vertices  $V_1\cup V_2$ and its edge set includes all the edges in the set
\[E_1\cup E_2\cup \{\{x,y\}~|~ x\in V_1,~y\in V_2\}.\]
 For example, the graph resulting from joining $K_2$ and $K_2$ is the complete graph $(K_4)$ of order 4. The graph resulting from joining the circle graph of $C_n$ and $K_1$ is  the wheel graph which  is denoted  by $W_n=C_n\vee K_1$. The fan graph is denoted by the symbol $G_n$ and is the join of the path  $P_n$ and $K_1$. In the figure \ref{fan} you can see the graph of $W_7$ and $G_7$. Next, using the tree-matrix theorem, we obtain the number of spanning trees of the wheel graph (\cite{Koshy}). We need the following theorem, known as Binet's theorem, which was obtained in 1843 by the French mathematician Jacques-Philippe Marie Binet.

\begin{theorem}\label{Binet}{\rm\cite{Koshy}}
	If $\alpha$ and $\beta$ are two roots of the equation $x^2-x-1=0$, then
	\begin{itemize}
		\item
		For each natural number $n$ we have:
		\[
		F_n=\frac{\alpha^n-\beta^n}{\alpha-\beta}=\frac{1}{\sqrt{5}}(\phi^n-(1-\phi)^n) .
		\]
		\item
		For each natural number $n$ we have:
		\[
		L_n=\alpha^n+\beta^n=\phi^n+(1-\phi)^n.
		\]
	\end{itemize}
\end{theorem}

\begin{theorem}\label{nw}{\rm\cite{Koshy}}
	The number of spanning trees of the wheel graph $W_n$ for $n\geq 3$ is
		\[
	\tau(W_n)=\phi^{2n}+(1-\phi)^{2n}-2.
	\]
\end{theorem}
\begin{proof}
	If we denote the central vertex of the wheel graph by the symbol $v_{n+1}$ and the rest of the graph vertices by $v_1, ..., v_n$, then for $i\neq n+1$ we have
	$\deg(v_i)=3$ and $\deg(v_{n+1})=n$. According to the tree-matrix theorem, we consider the following matrix:
	
	{
		\begin{eqnarray*}
			D(W_n)-A(W_n)&=&
			\left(\begin{array}{ccccccc}
				3& 0&0&0&\ldots &0&0 \\
				0& 3 &0 &0&\ldots &0&0 \\
				0& 0& 3 &0 & \ldots &0&0 \\
				&& & & \ldots && \\
				&& & & \ldots && \\
				0 &0& 0 &0 & \ldots &&n \\
			\end{array}\right)
			-
			\left(\begin{array}{ccccccc}
				0& 1&0&0&\ldots &0&1 \\
				1& 0 &1 &0&\ldots &0&0 \\
				0& 1& 0 &1 & \ldots &0&0 \\
				&& & & \ldots && \\
				&& & & \ldots && \\
				0 &0& 0 &0 & \ldots &1&0 \\
			\end{array}\right)\\
			&=&
			\left(\begin{array}{cccc|c}
				&&&&-1 \\
				&&A_n&&-1 \\
				&&&&\vdots \\
				&&&&-1 \\
				\hline
				-1&-1&\ldots &-1&n\\
			\end{array} \right),
		\end{eqnarray*} }
		Now, it is enough to obtain the cofactor of arbitrary entry of the matrix $D(W_n) - A(W_n)$. Here, we consider the entry in row $(n+1)$ and column $(n+1)$. We  have:
		\[
		\tau(W_n)=\det(A_n)=L_{2n}-2.
		\]
		Now, using  Theorem \ref{Binet}, we have the result. \qed
	\end{proof}

	Now, using the number of spanning trees of the wheel graph, we find the number of spanning trees of the fan graph:
	
	\begin{theorem}
		The number of spanning trees of the fan graph $G_n$ is equal to
		\[
		\tau(G_n)=\phi^{2n-1}+(1-\phi)^{2n-1}.
		\]
	\end{theorem}
	\begin{proof}
		According to Theorem \ref{span} for the wheel graph, we have:
		$\tau(W_n) = \tau(G_n) + \tau(W_{n-1})$. So $\tau(G_n)=\tau(W_n)-\tau(W_{n-1})$ and the result is follows by Theorem \ref{nw}. \qed
			\end{proof}

\section{Topological indices and $\phi$ }

The topological index of a graph $G$ is a real number dependent on the graph $G$, which represents a molecular chemical property whose molecular graph is $G$.
This number does not depend on the structure or appearance of the graph. In the chemical and mathematical graph theory of chemistry, a topological index, also known as a connectivity index, is a type of molecular descriptor that is calculated based on the corresponding graph of the molecule of a chemical compound.
Topological indices have many applications as tools for modelling the chemical properties of molecules. Wiener's index is one of the main studied topological indices in terms of theory and application. This index was the first topological index used in chemistry.
The boiling point of organic compounds, as well as all their physical properties, is functionally dependent on their number, type, and structure. The arrangement of atoms in a molecule in a group of isomers, the number and type of atoms are constant, and there are changes in physical properties only due to changes in structural interactions. American chemist Harold Weiner \cite{Wiener}
proved that the boiling point of paraffins with the sum the distances between carbons are related. He defined this index as $W(G)=\frac{1}{2} \sum_{\{x,y\}\subseteq V(G)} d(x,y)$,  where $d(x,y)$ is the length of the shortest  path between two vertices $x$ and $y$.

The Hosoya index of graph $G$ was introduced by Haruo Hosoya  in 1971 \cite{Hosoya1}. This
 index equals the sum of the number of $k$-matching of graph $G$  and is denoted by the symbol $Z(G)$. In other words
$Z(G)=\sum_{k=0}^{\alpha'(G)}m(G, k)$, where $m(G,k)$ denotes the number of matching sets of size $k$.
This index is often used in chemistry to check organic compounds.

Among the oldest and most famous topological indices, there are two famous topological indices based on vertex degree that we consider here. The first Zagreb index and the second Zagreb index.
Zagreb indices were introduced more than forty years ago by Gutman and Trinajestic \cite{Diu}.
The first and second Zagreb indexes of graph $G$ are denoted by $M_1(G)$ and $M_2(G)$, respectively, and they are:
\[
M_1=M_1(G)=\sum_{v \in V(G)} d^2(v)\label{s08},
\]
\[
M_2=M_2(G)=\sum_{uv \in E(G)} d_G(u)d_G(v)=\sum_{uv \in E(G)} (d(u)+d(v)),
\]
where $d(u)$ is the degree of vertex $u$ in graph $G$.

The generalized of the first Zagreb index of the graph $G$ is denoted by $M_1^k(G)$ and 
is defined as follows:
$$M_1^k(G)=\sum_{u\in V(G)} d(u)^k.$$
Note that:
\begin{enumerate}
	\item $M_1^1(G)=2|E(G)|$.
	\item $M_1^2(G)=M_1(G)$.
	\item $M_1^3(G)=F(G),$
\end{enumerate}
where  $F(G)$ is called the forgotten index.

Here, we review  the relationship between the generalized Zagreb index $M_1^k(G)=\sum_{u\in V(G)} d(u)^k$ and the golden ratio.
The following theorem is a rewrite of Theorem 1 of  \cite{Knox} in the form of topological indices. In other words, the following theorem gives an upper bound for the $M_2(G)$ in terms of the sum of all vertex degrees to the power of $\phi^2$.

\begin{theorem} \label{thm1}
	For each natural number $k$ and each graph $G$ with $\Delta(G)\leq 2k$ we have:
	$$M_2(G)\leq k^{2-\phi}M_{\phi^2}(G).$$
\end{theorem}

Knox, Mohar and Wood \cite{Knox} showed by proving the next lemma (which we wrote as a topological index) that the power $\phi^2$ and the constant $k^{\phi-2}$ in the theorem \ref{thm1} cannot be improved.

\begin{lemma}
	For any natural number $k$ and $\epsilon>0$, there exists a graph $G$ with $\Delta\leq 2k$ such that
	\[
	(1+\epsilon) M_2(G)\geq k^{2-\phi}M_{\phi^2}(G).
	\]
\end{lemma}

Huang, Shi and Xu proved the following interesting theorem for Hosoya index.

\begin{theorem}\label{main}{\rm\cite{Huang}}
For a graph $G$ of size $m$, 
	\[
	m+1 \leq Z(G) \leq \frac{1}{\sqrt{5}} \bigg[ \phi^{m+2}- (-\phi^{-1})^{m+2 }\bigg].
	\]
\end{theorem}

\section{The simplex graph of a graph and $\phi$}
In this section we review the relationship between the simplex graph of a graph and the golden ratio. 
Simplex graphs were introduced by Bandelt and van de Vel \cite{Band}. They  showed that the chromatic number of the underlying graph equals the minimum number $n$ such that the simplex graph can be isometrically embedded into a Cartesian product of $n$ trees.
Imrich, Klav\v zar and Mulder  also used the simplex graphs as part of their proof that testing whether a graph is triangle-free or whether it is a median graph may be performed equally quickly (\cite{Imrich}). 
Let recall the following definition: 

\begin{definition}
	 The simplex graph $\kappa(G)$ of an undirected graph $G$ is itself a graph, with one node for each clique (a set of mutually adjacent vertices) in $G$. Two vertices  of $\kappa(G)$ are linked by an edge whenever the corresponding two cliques differ in the presence or absence of a single vertex.
	\end{definition}

\begin{figure}[h]
	\begin{center}
	\includegraphics[width=0.6\textwidth]{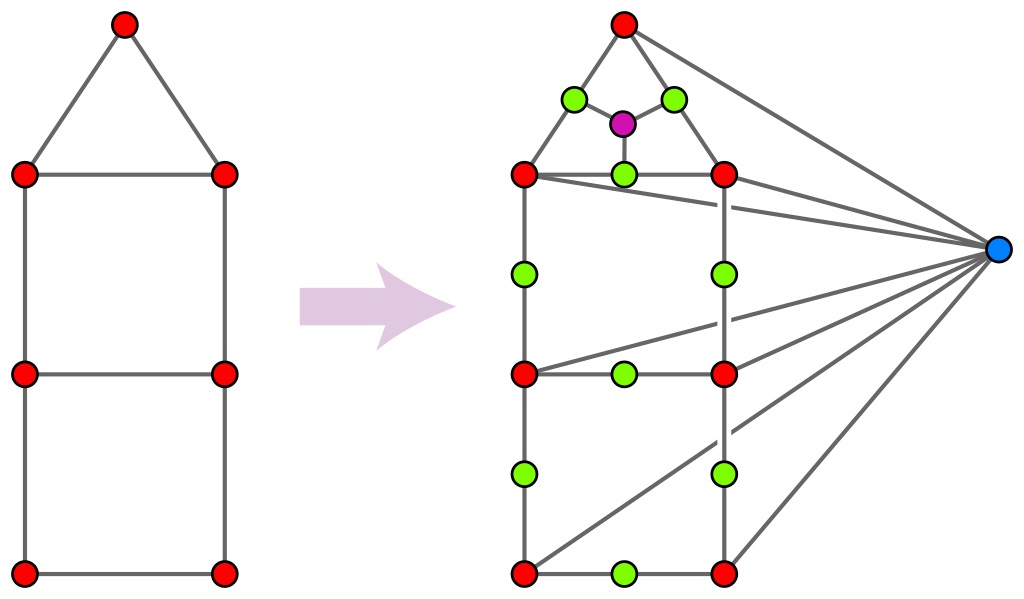}
	\caption{Graph $G$ and its simplex graph}
	\end{center} 
\end{figure}

\begin{example}
The simplex graph of the complete graph $K_n$, is a hypercube graph. In Figure \ref{Q4} you can see the cubic graph of $Q_4$.
	Also in the form of Figure \ref{cube},
	You can see graphs $Q_2$, $Q_3$ and $Q_4$.
	
\end{example}
\begin{figure}[h]
	\begin{center}
	\includegraphics[width=0.4\textwidth]{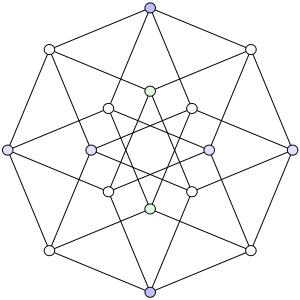}
	\caption{cubic graph $Q_4$}\label{Q4}
	\end{center} 
\end{figure}

\begin{figure}[h]
\begin{center} 
		\hspace{2cm}
	\includegraphics[width=0.6\textwidth]{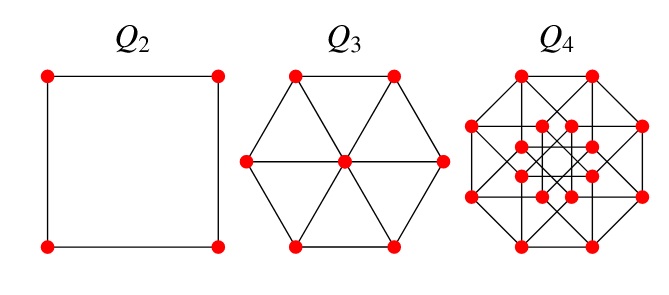}
	\caption {cubic graphs $Q_2, Q_3$ and $Q_4$.}
	\end{center} 
\end{figure}\label{cube}

\begin{theorem}
		The simplex graph of every graph is a bipartite graph.
\end{theorem}

	\begin{proof}
		To prove that the simplex graph is bipartite, it is enough to place the vertices that correspond to odd-sized cliques in the $X$ section, 
		and vertices corresponding to cliques of even size in section $Y$.  In this case, according to the definition of the simplex graph, no two vertices in the $X$ section (as well as in the $Y$ section) are connected, and therefore the graph is bipartite.
	\qed
	\end{proof}
	
	Like the hypercube graph, the vertices of the Fibonacci cube of order $n$ may be labeled with bitstrings of length $n$, in such a way that two vertices are adjacent whenever their labels differ in a single bit. However, in a Fibonacci cube, the only labels that are allowed are bitstrings with no two consecutive $1$ bits. If the labels of the hypercube are interpreted as binary numbers, the labels in the Fibonacci cube are a subset, the fibbinary numbers (\cite{Ard}). There are $F_{n+2}$ labels possible, where $F_n$ denotes the $n$th Fibonacci number, and therefore there are $F_{n + 2}$ vertices in the Fibonacci cube of order $n$.

	\begin{theorem} 
		The simplex graph of the complement of path graph $P_n$ is Fibonacci cube.
		\end{theorem}

	It is clear that the simplex graph of  $\kappa(G)$, can be denoted by the symbol $\kappa^2(G)$ and inductively the $n$-simplex graph $G$ can be denoted by the symbol $\kappa^n(G)$ and defined as $\kappa^n(G)=\kappa(\kappa^{n-1}(G))$. According to the structure of this graph, we can find recurrence relations for the number of vertices and edges $\kappa^n(G)$. For convenience, we set $v_n=|V(\kappa^{n}(G))|$ and $e_n=|E(\kappa^{n}(G))|$. Note that $v_0=|V(G)|$ and $e_0=|E(G)|$. We have the following return relations:
	
	\begin{theorem}\label{rec}
		For each $n\geq 2$ we have:
		\begin{enumerate}
			\item[]
			$v_n=3v_{n-1}-v_{n-2}-1.$
			\item[]
			$e_n=3e_{n-1}-e_{n-2}+1.$
		\end{enumerate}
	\end{theorem}
	\begin{proof}
		
		First, note that according to the definition of the simplex graph, the largest clique in $\kappa(G)$ is equal to $2$-clique, or $K_2$. Since the number of cliques in $\kappa(G)$ equals 
		$|V(\kappa(G))|+|E(\kappa(G))|+1$
		or is $v_1+e_1+1$, it can be seen by induction that
		\[
		v_n=v_{n-1}+e_{n-1}+1.
		\]
		And according to the structure of the graph $\kappa(G)$ we have:
		\[
		e_n=2e_{n-1}+v_{n-1}.
		\]
		By using these two relationships, both relationships can be easily obtained. \qed 
	\end{proof}
	
	Recursive relations in the Theorem \ref{rec}  are the second-order recursive relations, whose general answer is as follows:
	\begin{eqnarray}
		v_n=c_1(1+\phi)^n+c_2(2-\phi)^n+1.
	\end{eqnarray}
	\begin{eqnarray}
		e_n=d_1(1+\phi)^n+d_2(2-\phi)^n-1.
	\end{eqnarray}
	By inserting the initial conditions, the following answers are obtained:
	
	{
		\begin{eqnarray*}
			v_n&=&\frac{\big(v_2+1-\phi-v_1(2-\phi)\big)}{\sqrt{5}}(1+\phi)^{n-1}+
			\frac{\big(v_1(1+\phi)-v_2-\phi)\big)}{\sqrt{5}}(2-\phi)^{n-1}+1.
		\end{eqnarray*} }
		
		{
			\begin{eqnarray*}
				e_n&=&\frac{\big(e_2-(1-\phi)-e_1(2-\phi)\big)}{\sqrt{5}}(1+\phi)^{n-1}+
				\frac{\big(e_1(1+\phi)+e_2\phi)\big)}{\sqrt{5}}(2-\phi)^{n-1}-1.
			\end{eqnarray*} }
			Since $0<2-\phi<1$ for very large $n$ the power of this number
			 tends to zero and therefore the ratio $\frac{e_n}{v_n}$ for large $n$ is equal to the following fraction:
			\[
			\frac{e_2- (1-\phi)-e_1(2-\phi)}{v_2+1-\phi-v_1(2-\phi)}.
			\]
			By putting the relations $v_2=v_1+e_1+1$ and $e_2=2e_1+v_1$, the above ratio becomes as follows:
			\[
			\frac{v_1+e_1\phi- (1-\phi)}{e_1+2-\phi-v_1(1-\phi)}.
			\]
			If we multiply the denominator of this fraction in $\phi$, its numerator is obtained. Therefore, we will have the following theorem:
			
\begin{theorem}
	For any graph $G$, 
	\[
	{\lim}_{n\rightarrow \infty}\frac{|E(k^{n}(G))|}{|V(k^{n}(G))|}=\phi.
	\]
\end{theorem}

\section{Energy of graphs and $\phi$}
If $A$ is the adjacency matrix of $G$, then
the eigenvalues of $A$, $\lambda_1\geq \lambda_2\geq ...\geq \lambda_n$ are said to be the eigenvalues of
the graph $G$. These are the roots of the characteristic polynomial $\phi(G,\lambda)=\prod_{i=1}^n(\lambda-\lambda_i)$.
It is easy to see that the golden ratio $\phi$ can be eigenvalue of graph. For example the eigenvalues of graph $P_4$  are $\pm\phi, \pm\phi^{-1}$.  The graph $P_4$ is the smallest graph of the form $H\circ K_2$, where $H=K_2$ and $\circ$ is the corona product. For more information on eigenvalues of corona of two graphs see \cite{SIAM}.  

The energy of the graph $G$ is defined as $E=E(G)=\sum_{i=1}^n |\lambda_i|$. This
definition was put forward by I. Gutman \cite{Gutman} and was motivated by earlier results in theoretical chemistry \cite{Pol}. 

In 2004 Bapat and Pati \cite{Bapat} obtained the following result:
\begin{theorem} 
 The energy of a graph cannot be an odd integer.
 \end{theorem} 
In 2008 Pirzada and Gutman communicated an interesting result:
\begin{theorem}{\rm\cite{Pir}}
 The energy of a graph cannot be the square root of an odd integer.
 \end{theorem} 
 
 In 2010, Alikhani and Iranmanesh proved the following result. 
 \begin{theorem}{\rm\cite{IJMSI}}
 The golden ratio cannot be the energy of a graph.
 \end{theorem} 

\section{Chromatic roots and domination roots and $\phi$} 

Let $G=(V,E)$ be a simple graph and $\lambda \in \mathbb{N}$.
A mapping $f\colon V\rightarrow \{1,2,\ldots,\lambda\}$ is called a {\it $\lambda$-colouring} of $G$,
if $f(u)\neq f(v)$ whenever the vertices $u$ and $v$ are adjacent in $G$.
The number of distinct $\lambda$-colourings of $G$, denoted by $P(G,\lambda)$
is called the {\it chromatic polynomial} of $G$.
A root of $P(G,\lambda)$ is called a
{\it chromatic root} of $G$.
An interval is called a {\it root-free interval}
for a chromatic polynomial $P(G,\lambda)$, if $G$ has no chromatic root in this interval.
It is well-known that $(-\infty,0)$ and $(0,1)$ are two maximal root-free intervals for
the family of all graphs (see~\cite{jackson1}). Jackson \cite{jackson1} showed that
$(1,\frac{32}{27}]$ is another maximal root-free interval for the family of all graphs and
the value $\frac{32}{27}$ is best possible. The chromatic polynomial, chromaticity of graphs and chromatic zeros studied well in 
\cite{Dong}.

 Tutte \cite{tutte1} proved that $\frac{3+\sqrt{5}}{2}=1+\phi=\phi^2$  cannot be a chromatic root. Alikhani and Peng proved the following result: 
 
 \begin{theorem}{\rm\cite{AADM1}}\label{pn}
  $\phi^{n} (n\in \mathbb{N})$ cannot be  roots of any chromatic polynomials.
\end{theorem}

We need the following theorems:
\begin{theorem} {\rm\cite{Root}}\label{plus}
 If $\alpha$  is a chromatic root, then for any natural number $n$, $\alpha+n$ is a chromatic root.
\end{theorem} 
\begin{theorem} {\rm\cite{Root}}\label{three}
$\phi+3$ is a chromatic root.
\end{theorem} 

By Theorems \ref{plus} and \ref{three} we have the following corollary: 

\begin{corollary}
For every natural number $n\geq 3$, $\phi+n$ is a chromatic root.
\end{corollary} 
Note that since $\phi+1=\phi^2$, by Theorem \ref{pn}, $\phi+1$ cannot be a chromatic roots. Harvey and Royle \cite{JGT} proved that $\phi+2$ is a chromatic root.
Two graphs in Figure \ref{tt2} have $\phi+2$  as a chromatic root.

\begin{figure}[h]
	\begin{center} 
		\includegraphics[width=0.7\textwidth]{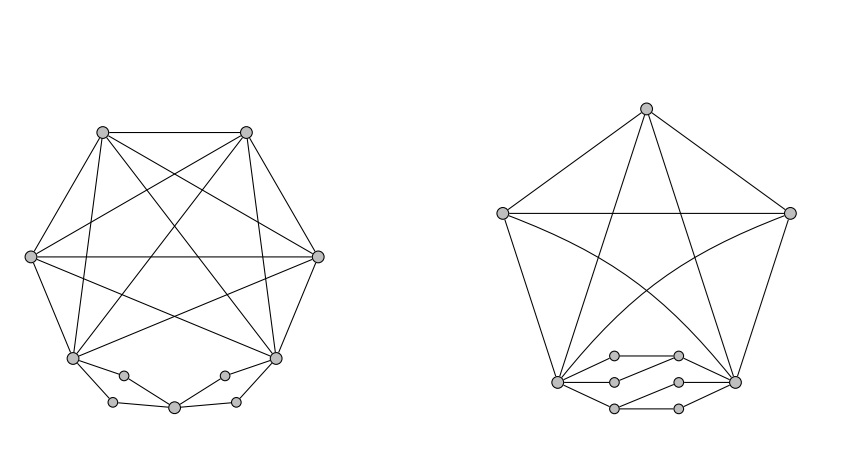}
		\caption{Two graphs  with $\phi+2$  as a chromatic root.}\label{tt2}
	\end{center} 
\end{figure}

Alikhani and Peng  studied the $n$-anacci constants and their powers as chromatic roots \cite{AADM2}. 
They  proved that the $2n$-anacci numbers cannot be zeros of any
chromatic polynomials.

\begin{theorem} \label{theorem8}{\rm\cite{AADM2}}
	For every integer $n\geq1$, the $2n$-anacci number $\varphi_{2n}$ and all natural powers of them cannot be zero of any chromatic polynomial.
\end{theorem}

 There is a conjecture  in \cite{AADM2} which state that  $(2n + 1)$-anacci numbers and all their natural powers also cannot be chromatic zeros. They obtained the following result related to $\varphi_{2n+1}$.
 
 \begin{theorem}\label{theorem11}{\rm\cite{AADM2}}
 	For every natural $n$, $\varphi_{2n+1}$ cannot be a root of the
 	chromatic polynomial of a connected  graph $G$ with $|V(G)|\leq
 	4n+2$.
 \end{theorem}

  \medskip
  
  A set $S\subseteq V$ is a dominating set of a graph $G$, if every vertex in $V \setminus S$  has at least one neighbor in $S$. The cardinality of a minimum  dominating set in $G$ is called the domination number of $G$ and is denoted by $\gamma(G)$. 
 The various different domination concepts are
 well-studied now, however new concepts are introduced frequently and the interest is growing
 rapidly. We recommend excellent so-called domination books by Haynes, Hedetniemi, and Slater \cite{7,8}.

Domination polynomial of a graph $G$ is $D(G,x)=\sum d(G,k)x^k$ is the generating function for the number of dominating sets of $G$. Note that $d(G,k)$ is the number of dominating sets of $G$ of size $k$. A root of $D(G,x)$  is
called a domination root of $G$. The set of distinct domination roots of graph $G$ is denoted 
by $Z(D(G,x))$.

For domination polynomial of a graph, it is clear that $(0,\infty)$  is zero-free interval. Brouwer \cite{Brou}  has shown that $-1$ cannot be domination root of any graph $G$. For more details of the domination polynomial of a graph at $-1$ refer to \cite{GCOM}.
It also have shown that every integer domination root is even \cite{GCOM}. 
Alikhani and Hasni proved the following theorem:  

\begin{theorem}{\rm\cite{Root}}
If $n$ is an odd natural number, then $-\phi^n$ cannot be domination roots.
\end{theorem}

Let us recall the corona of two graphs. The corona of two graphs $G_1$ and $G_2$, as defined
by Frucht and Harary in \cite{Frucht}, is the graph $G_1\circ G_2$  formed from one copy of $G_1$ and
$|V(G_1)|$  copies of $G_2$, where the $i$th vertex of $G_1$ is adjacent to every vertex in the $i$th copy of
$G_2$.

\begin{theorem}\label{theorem5}{\rm\cite{euro}}
	Let $G$ be a connected graph of order $n$. Then,
	$Z(D(G,x))= \{0, \frac{-3\pm \sqrt{5}}{2}\}$, if and only if
	$G=H\circ \overline{K_2}$, for some graph $H$.
	Indeed $D(G,x)=x^{\frac{n}{3}}(x^2+3x+1)^{\frac{n}{3}}$.
\end{theorem}
The following corollary is an immediate consequence of Theorem \ref{theorem5}.

\begin{corollary} 
 All graphs of the form $H\circ \overline{K_2}$, have $-\phi^2$ as domination roots.
 \end{corollary}

\section{The number of domatic partition and $\phi$ }

 A domatic partition is a partition
 of the vertex set into dominating sets, in other words, a partition $\pi$ = $\{V_1, V_2, . . . , V_k \}$ of $V(G)$  
 such that every set $V_{i}$ is a dominating set in $G$. 
 Cockayne and Hedetniemi \cite{Cockayne} introduced the domatic number of a graph $d(G)$ as the maximum order $k$ of a vertex partition. For more details on the domatic number refer to e.g., \cite{11,12,13}.

 Motivated by enumerating of  the number of dominating sets of a graph and domination polynomial (see e.g. \cite{euro}), the enumeration of  the domatic partition for certain
 graphs is a natural subject.

 \begin{definition}\label{def}
 	Let ${\cal DP}(G,i)$ be the family of
 	domatic partition of a graph $G$ with cardinality $i$, and let
 	$dp(G,i)=|{\cal DP}(G,i)|$. 
 \end{definition}

In the following we show that 
$$\lim_{n \to \infty}\frac{dp(P_{n+1},2)}{dp(P_{n},2)}=\phi.$$

\begin{figure}[h]
	\begin{center} 
		\includegraphics[width=0.7\textwidth]{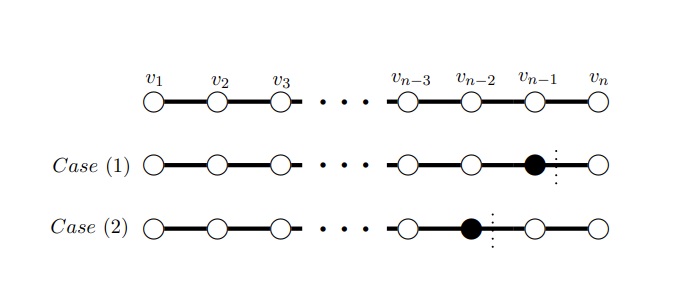}
		\end{center}
	\caption{Path graph and possible cases related to the Theorem \ref{thm:path-domatic}.} \label{fig:path-domatic}
\end{figure}


\begin{theorem}\label{thm:path-domatic}
	For $n\geq 4$, $dp(P_n,2)=dp(P_{n-1},2)+dp(P_{n-2},2)$, where $dp(P_2,2)=dp(P_3,2)=1$. 
\end{theorem}

\begin{proof}
	It is easy to see that $dp(P_2,2)=dp(P_3,2)=1$. Now, suppose that $n\geq 4$ and  let $V(P_n)=\{v_1,v_2,\ldots,v_n\}$ as we see in figure \ref{fig:path-domatic}. Suppose that $D=\{A,B\}$ is a domatic partition of $P_n$. If $\{v_{n-1},v_{n}\}\subseteq A$, then $D$ is not a domatic partion, because then $B$ is not a dominating set, since $v_n$ is not dominated by any other vertices of $B$. So without loss of generality, we have $v_{n-1}\in A$ and $v_{n}\in B$. To show that $dp(P_n,2)=dp(P_{n-1},2)+dp(P_{n-2},2)$, consider the following cases:
	\begin{itemize}
		\item[$Case ~(1).$]
		If we consider $P_{n-1}$, and let $D_1=\{A_1,B_1\}$ a domatic partion of $P_{n-1}$, then without loss of generality suppose that $v_{n-1}\in A_1$. Then by our argument, $v_{n-2}\in B_1$. In this case let, $A=A_1$ and $B=B_1\cup\{v_n\}$. Then it is clear that $D=\{A,B\}$ is a domatic partition of $P_n$, and we have $dp(P_{n-1},2)$ domatic partition of size 2 for $P_n$.
		\item[$Case ~(2).$]
		If we consider $P_{n-2}$, and let $D_2=\{A_2,B_2\}$ a domatic partion of $P_{n-2}$, then without loss of generality suppose that $v_{n-2}\in A_2$. Then by our argument, $v_{n-3}\in B_2$. In this case let, $A=A_2\cup\{v_{n-1}\}$ and $B=B_2\cup\{v_{n}\}$. Then it is clear that $D=\{A,B\}$ is a domatic partition of $P_n$, and we have $dp(P_{n-2},2)$ domatic partition of size 2 for $P_n$.
	\end{itemize}
	Clearly, the domatic partitions from Case~(1) and Case~(2) are different. We claim that we have counted all the possible cases and only need to know what is happening to $v_{n-1}$ and $v_n$ for domatic partitions of $P_n$, and therefore we don't need to see what happens in other cases and other cases are included in the mentioned ones. 
	To show our claim, suppose that we continue our cases. If we consider $P_{n-3}$, and let $D_3=\{A_3,B_3\}$ a domatic partion of $P_{n-3}$, $v_{n-3}\in A_3$, and $v_{n-4}\in B_3$, then to make  dominating sets from $A_3$ and $B_3$ to have domatic partition, we need to consider $A=A_3\cup\{v_{n-1}\}=A_1$ and $B=B_3\cup\{v_{n-2},v_{n}\}=B_1\cup\{v_n\}$ which is calculated in Case~(1), or $A=A_3\cup\{v_{n-2},v_{n}\}=B_1\cup\{v_n\}$ and $B=B_3\cup\{v_{n-1}\}=A_1$ which is calculated in Case~(1) too. By a similar argument, we conclude that we have counted all domatic partitions for other cases too.
	Note that if a dominating set has three consequent vertices, then its complement is not a dominating set. Therefore we have the result. 
	\qed
\end{proof}

As an immediate result of Theorem \ref{thm:path-domatic}, we have:

\begin{corollary}
	$\lim_{n \to \infty}\frac{dp(P_{n+1},2)}{dp(P_{n},2)}=\phi.$ 
\end{corollary}

\begin{proof}
	Since $dp(P_n,2)$ follows the Fibonacci sequence, we have the result.\qed
\end{proof}	
 
 \medskip

In \cite{domaticp} we extend this result to weak $2$-coloring of a graph. Let recall the following definition: 

\begin{definition}{\rm\cite{domaticp}}
A weak $k$-coloring of a graph $G=(V,E)$ assigns a color $c(v)\in \{1,2,\ldots,k\}$ to each vertex $v\in V$, such that each non-isolated vertex is adjacent to at least one vertex with different color. 
\end{definition} 
So a weak $2$-coloring of a graph is equivalent to finding a domatic partition of a graph of size $2$. In the following, let ${\cal W}(G,2)$ be the family of
weak $2$-coloring of a graph $G$, and let
$w_2(G)=|{\cal W}(G,2)|$. So $dp(G,2)=w_2(G)$.

 \begin{theorem}{\rm\cite{domaticp}}
 	\label{co1}
 	For any path $P_n$ of order $n\geq 4$, 
 	$$w_2(P_n)=w_2(P_{n-1})+w_2(P_{n-2}),$$ 
 	where $w_2(P_3)=w_2(P_2)=1$.
 \end{theorem}
 As an immediate result of Corollary \ref{co1}, we have:
 \begin{corollary}
 	$\lim_{n \to \infty}\frac{w_2(P_{n+1})}{w_2(P_n)}=\phi.$
  \end{corollary}


\end{document}